\documentclass[11pt]{article}
\usepackage{dsfont}
\usepackage{amsmath, amsthm, amssymb, mathrsfs, eufrak}
\usepackage[abbrev,non-sorted-cites]{amsrefs}
\usepackage[all]{xy}
\usepackage{CJK}
\usepackage{grffile}
\usepackage{graphicx}
\newtheorem{thm}{Theorem}[section]

\newtheorem{lem}[thm]{Lemma}

\newtheorem{definition}[thm]{Definition}
\newtheorem{rmk}[thm]{Remark}

\newtheorem{question}[thm]{Question}

\begin{document}

\title{On the Tropical Discs Counting on Elliptic K3 Surfaces with General Singular Fibres}
\author{Yu-Shen Lin}
\maketitle
\section*{Abstract}
Using Lagrangian Floer theory, we study the tropical geometry of K3 surfaces with more general singular fibres other than simple nodal curves. In particular, we give the local models for the type $I_n$, $II$, $III$ and $IV$ singular fibres in the Kodaira's classification. This generalizes the correspondence theorem in \cite{L8} between open Gromov-Witten invariants/tropical discs counting to these cases.
\section{Introduction}

The notion of tropical geometry is invented in honor of the Brazilian mathematician Imre Simon who pioneered the studies of min-plus algebras.
Tropical geometry later becomes an important branch of algebraic geometry and often transforms algebraic geometric problems into combinatoric ones. Enumerative geometry benefits much from tropical geometry. Tropical geometry occurs naturally when certain worst possible degeneration happens, which is known as the maximal unipotent point. The Strominger-Yau-Zaslow degeneration \cite{SYZ}\cite{KS4} exactly describes such scenario: the special Lagrangian fibrations of the Calabi-Yau manifolds collapse to their base affine manifolds when the Calabi-Yau manifolds approaching the maximal unipotent points. Moreover, the holomorphic curves in the Calabi-Yau manifolds degenerate into union of straight line segments on the base affine manifolds. These degenerate $1$-skeletons are known as the tropical curves. The enumeration of the holomorphic curves are expected to be computed via the enumeration of tropical curves on the affine manifolds. Mikhalkin \cite{M2} first achieved the success in this aspect by proving that the counting of Riemann surfaces with incidence conditions in toric surfaces can be computed by the weighted counting of tropical curves on $\mathbb{R}^2$\footnote{One view $\mathbb{R}^2$ as the Legendre dual of the interior of the moment polytope.} with the corresponding incidence conditions. The correspondence theorem is later generalized to toric manifolds for all dimension (but only for genus zero curves) by Nishinou-Siebert \cite{NS}. 

Following the pioneering work of Kontsevich-Soibelman \cite{KS1}\cite{KS2} and Gross-Siebert \cite{GS1} and Fukaya \cite{F3}, the author defined the notion of tropical discs when the K3 surfaces with special Lagrangian fibration admit only type $I_1$ singular fibres, namely immersed singular fibres \cite{L8}. It is well-known that the correct definition of tropical discs can go into the type $I_1$ singularities from certain two directions\cite{KS1}\cite{G3}. Indeed, the local relative homology around a type $I_1$ singular fibre is generated by the associated Lefschetz thimble. The Lefschetz thimble (up to a sign) determines the direction that a tropical discs can go into the singularity. The author also defined the notion of admissible ones which contributes to the weighted counting of the tropical discs counting and each admissible tropical disc is assigned with a weight \cite{L8}. The tropical discs counting $\tilde{\Omega}^{trop}$ is then the weighted counting of the admissible tropical discs. On the other hand, using the similar idea of reduced Gromov-Witten invariants in algebraic geometry, the author also defined the open Gromov-Witten invariants on K3 surfaces \cite{L4}. Moreover, a correspondence theorem which connects the open Gromov-Witten invariants and the tropical discs counting is established in \cite{L8} when the special Lagrangian fibration admits only type $I_1$ singular fibres.  Although generic elliptic K3 surfaces admit only type $I_1$ singular fibres (see Theorem \ref{35}), the type $I_1$ singular fibres may emerge together and develop other type of singularities. Since any special Lagrangian fibration in a K3 surface becomes an elliptic fibration after hyperK\"ahler rotation. All the possible singular fibres are classified by Kodaira \cite{K1}. Especially for the singularities other than the type $I_n$ singularity, all the $1$-cycles of the torus fibres are "vanishing cycles". In particular, we want to answer the questions:
\begin{question} 
  What are the directions a tropical discs can go into a singularity that appears in the Kodaira's list?
\end{question}
\begin{question}
  What are the corresponding open Gromov-Witten invariants? 
\end{question}
The second question provides the "initial condition" for the wall-crossing that governs the construction of the mirror manifold. In this paper, we will explore the corresponding tropical geometry when the special Lagrangian fibration admits type $I_n$, $II$, $III$, $IV$ and $I_0^*$ singular fibres in the classification. 

The article is arranged as follows: In Section 2, we review the definition and properties of the open Gromov-Witten invariants on hyperK\"ahler surfaces with elliptic fibration developed in \cite{L4}\cite{L8}. In Section 3, we review the tropical geometry on K3 surfaces with only type $I_1$ singular fibres. In Section 4, we discuss the tropical geometry in the presence of type $I_n$, $II$, $III$, $IV$ type singular fibre. We also have some partial answer to the case of type $I_0^*$ singular fibres.

\section*{Acknowledgement}
The author would like to thank Shing-Tung Yau for constant support and encouragement. The author want to thank Chiu-Chu Melissa Liu, Tom Sutherland, Hansol Hong for helpful discussions. The author would also like to thank the referees for the useful comments and suggestions. Part of the work is done during the period the author is in Columbia University. The author is supported by the Center of Mathematical Sciences and Applications at Harvard University at the time of submitting the paper. 

\section{Open Gromov-Witten Invariants on K3 Surfaces}
\subsection{HyperK\"ahler Geometry}
   \begin{definition}
     \begin{enumerate}
        \item A holomorphic symplectic $2$-form is a $d$-closed, non-degenerate holomorphic $2$-form.
        \item A K\"ahler manifold $X$ of dimension $n$ is a hyperK\"ahler manifold if there exists a covariant constant holomorphic symplectic $2$-form.            
     \end{enumerate}
     \end{definition}
Let $\Omega$ denote the covariant constant holomorphic symplectic  form and $\omega$ be the K\"ahler form, we have
            \begin{align*}
               \omega^n=\frac{(2n)!}{2^{2n}(n!)^2}\Omega^n\wedge \bar{\Omega}^n.
            \end{align*}  
The pair $(\omega,\Omega)$ then induces an $S^2$-family of complex structures determined by the holomorphic symplectic $2$-forms 
 \begin{align*}
 \Omega_{\zeta}&= -\frac{i}{2\zeta}\Omega+\omega-\frac{i}{2}\zeta\bar{\Omega}, \zeta\in \mathbb{P}^1 
 \end{align*}\footnote{Here we use the convention that $\Omega_0=\Omega$ and $\Omega_{\infty}=\bar{\Omega}$} and the corresponding K\"ahler forms are given by  
  \begin{align*}
     \omega_{\zeta}&= \frac{i(-\zeta+\bar{\zeta})\mbox{Re}\Omega-(\zeta+\bar{\zeta})\mbox{Im}\Omega+(1-|\zeta|^2)\omega}{1+|\zeta|^2}.
  \end{align*} In particular, when $\zeta=e^{i\vartheta}$, $\vartheta\in S^1$, we have
   \begin{align} \label{1001}
        \omega_{\vartheta}&:=\omega_{e^{i\vartheta}}= -\mbox{Im}(e^{-i\vartheta}\Omega)\notag \\
                  \Omega_{\vartheta}&:=\Omega_{e^{i\vartheta}}=\omega-\mbox{Re}(e^{-i\vartheta}\Omega).
   \end{align} It is straight-forward to check that the pair $(\omega_{\vartheta},\Omega_{\vartheta})$ determines a hyperK\"ahler structure on the same underlying space and we will denote the corresponding hyperK\"ahler manifold by $X_{\vartheta}$. 
   
   Let $L$ be a holomorphic Lagrangian in $X$, i.e., $\Omega|_L=0$. Then from the equation (\ref{1001}), $L$ is a special Lagrangian submanifold in $X_{\vartheta}$ and vice versa. In particular, the hyperK\"ahler manifold $X$ admits a holomorphic Lagrangian fibration if and only if $X_{\vartheta}$ admits a special Lagrangian fibration, for all $\vartheta\in S^1$. This is the so-called hyperK\"ahler rotation trick. 
   
   Let $X\rightarrow B$ be a hyperK\"ahler manifold with holomorphic Lagrangian fibration. Matsushita \cite{M} proved that $B$ is the projective space if it is smooth. Let $\Delta\subseteq B$ be the discriminant locus, then there exist integral affine structures on $B_0=B\backslash \Delta$. Indeed, consider the following function $Z$ called the central charge defined on a local system of lattices over $B_0$
      \begin{align*}
         Z:\bigcup_{u\in B_0}&H_2(X,L_u)\rightarrow \mathbb{C}\\ 
                         &\gamma_u\longmapsto Z_{\gamma}(u):=\int_{\gamma_u}\Omega.
      \end{align*} Choose local section $\gamma_i\in H_2(X,L_u)$ such that $\partial \gamma_i$ be the basis of $H_1(L_u)$, then
        \begin{align} \label{899}
        f_i(u):=\mbox{Re}(e^{-i\vartheta}Z_{\gamma_i}(u))
        \end{align}
         gives a set of integral affine coordinates on $B_0$. In dimension two, this is usually known as the complex affine coordinate induced by the special Lagrangian fibration on $X_{\vartheta}$ in the context of mirror symmetry. The integral affine structure on $B_0$ allows us to talk about tropical geometry in Section \ref{1005}. The holomorphicity of $\Omega$ implies that the central charge $Z$ is a (multi-value) holomorphic function on $B_0$.           
            
\subsection{Floer Theory on HyperK\"ahler Manifolds}
  Let $(M,\omega)$ be a symplectic manifold and $L$ be a Lagrangian submanifold. Choose a compatible complex structure $J$, Gromov \cite{G} started to explore the techniques of $J$-holomorphic discs to study the symplectic geometry of the pair $(M,L)$. One of important symplectic invariant is the Floer homology \cite{F7}. We will use the version constructed by Fukaya-Oh-Ohta-Ono \cite{FOOO}\cite{F1}:
    \begin{thm}
       There exists an $A_{\infty}$ structure on $H^*(L,\Lambda)$ constructed from the moduli spaces of holomorphic discs with boundary on $L$ which is unique up to pseudo-isotopies, where $\Lambda$ is the Novikov ring. 
    \end{thm} 
   Let $m_k:H^*(L,\Lambda)^{\otimes k}\rightarrow H^*(L,\Lambda)$ be the corresponding $A_{\infty}$ operators. The associated  Maurer-Cartan moduli space is defined to be
   \begin{align*}
      \mathcal{MC}(L):=\{b\in H^1(L,\Lambda_+)|\sum_k m_k(b,\cdots,b)=0\}/\sim,
   \end{align*} where $\sim$ denotes the gauge equivalence. It is a standard algebra fact that pseudo-isotopies between $A_{\infty}$ algebras induce isomorphisms between the associate Maurer-Cartan moduli spaces \cite{F1}. Let $\phi$ be a $1$-parameter family of tamed almost complex structures with respect to $\omega$, we will denote the isomorphism of Maurer-Cartan spaces induced by the pseudo-isotopy by $F_{\phi}:\mathcal{MC}(L)\rightarrow \mathcal{MC}(L)$. In particular, we will apply to the following  situation: Let $X\rightarrow B$ (special) Lagrangian fibration  with $\mbox{dim}_{\mathbb{R}}X=4$. Denote $L_u$ to be the fibre over $u\in B$. Fix a reference fibre $L_u$ and $u_+,u_-\in B_0$ near $u$ such that $L_{u_{\pm}}$ do not bound any holomorphic discs. Choose a path $\phi$ from $u_+$ to $u_-$ and a $1$-parameter family of diffeomorphism mapping $L_{\phi(t)}$ to $L_u$. This induces a $1$-parameter family of tamed complex structures on $X$ and in this case the Maurer-Cartan space $\mathcal{MC}(L_u)\cong H^1(L_u,\Lambda_+)$. We denote the corresponding isomorphism of Maurer-Cartan spaces by $F_{u,\phi}:H^1(L_u,\Lambda_+)\rightarrow H^1(L_u,\Lambda_+)$ (will drop the subindex $u$ if there is no confusion) with some abuse of notation. This is known as the Fukaya's trick. 
   
   The following lemma is from homological algebra and useful for the later argument. 
   \begin{lem} \label{521} \cite{F1}
           Assume that $\phi$ is a contractible loop, then $F_{\phi}=\mbox{Id}$. 
       \end{lem}
     Choose $e_1,e_2$ be basis of $H^1(L_u,\mathbb{Z})$. Given a general element $b=x_1e_1+x_2e_2\in H^1(L_u,\Lambda_+)$, we write $F_{\phi}(b)=(F_{\phi}(b))_1e_1+(F_{\phi}(b))_2e_2$. Inspired by the context of mirror symmetry, we denote $z_k=\exp{(F_{\phi}(b))_k}$ and $F_{\phi}$ induces an isomorphism of algebras, denoted by $\tilde{F}_{\phi}$:
         \begin{align}\label{520}
            \Lambda[[H_1(L_u)]]\rightarrow \Lambda[[H_1(L_u)]] \notag \\
           \tilde{F}_{\phi}:     z_k \mapsto \exp{(F_{\phi}(b))_k}  .
         \end{align}
     \begin{lem}
       The composition of the transformation is compatible with the transformation of composition of paths.
     \end{lem}    
      \begin{proof}
        Let $p_i:b\mapsto x_i$, then (\ref{520}) is equivalent to 
          \begin{align*}
              \tilde{F}_{\phi}(\exp{(p_i(b))})=\exp{(p_iF_{\phi}(b))}.
          \end{align*} Then we have that 
          \begin{align*}
             \tilde{F}_{\phi_2\circ \phi_1}(z_i)&=\exp{(p_i F_{\phi_2\circ\phi_1}(b))}  \\
             &=\exp{\big(p_iF_{\phi_2}(F_{\phi_1}(b))\big)}\\
             &=\tilde{F}_{\phi_2}\big(\exp{(p_iF_{\phi_1}(b))}\big)\\
             &=\tilde{F}_{\phi_2}\big(F_{\phi_2}(\exp{(p_i(b))})\big)=\tilde{F}_{\phi_2}\circ \tilde{F}_{\phi_1}(z_i). 
          \end{align*}
      \end{proof}  Together with Lemma \ref{521}, we have that $\tilde{F}_{\phi}$ also only depends on the homotopy class of $\phi$.     
      
   Now assume that the symplectic manifold $X$ is actually hyperK\"ahler with respect to $(\omega,\Omega)$ and admits a holomorphic Lagrangian fibration $p:X\rightarrow B$. From the discussion in previous section, $X_{\vartheta}\rightarrow B$ is a special Lagrangian fibration. Denote $L_u$ to be the smooth fibre over $u\in B_0$. Assume that $L_u$ bounds a holomorphic discs in $X_{\vartheta}$ in the relative class $\gamma\in H_2(X,L_u)$, then we have 
      \begin{align*}
         Z_{\gamma}(u)=e^{i(\vartheta-\frac{\pi}{2})}\int_{\gamma}\omega_{\vartheta} \mbox{ or } \mbox{Arg}Z_{\gamma}(u)=\vartheta-\frac{\pi}{2}.
     \end{align*} Since $dZ_{\gamma}(u)\neq 0$, this gives a real codimension one topological constraint to the existence of holomorphic discs. Explicitly, let $\gamma_i\in H_2(X,L_u)$ such that $\partial\gamma_i$ forms a basis of $H_1(L_u)$, then 
       \begin{align*}
          0=\mbox{Re}(e^{-i\vartheta}Z_{\gamma}(u))=a_1f_1(u)+a_2f_2(u)+\int_{\gamma_0}\mbox{Re}(e^{-i\vartheta}\Omega).
       \end{align*} if $\gamma=a_1\gamma_1+a_2\gamma_2+\gamma_0$ for some $\gamma_0\in H_2(X)$. In other words, if $\phi(t)$ is a path on $B_0$ such that $L_{\phi(t)}$ bounds holomorphic discs in $X_{\vartheta}$, then $\phi(t)$ falls in an affine line, which we will denote it by $l_{\gamma}$, in $B_0$. The affine lines $l_{\gamma}$ are the building blocks of the tropical discs. On the other hand, this implies that a generic fibre $L_u$ does not bound any holomorphic discs.
       
       The Cauchy-Riemann equation for $Z$ implies that the affine line is also characterized as the solution of the gradient flow lines (with respect to a K\"ahler metric on $B$)
        \begin{align*}
          & \frac{d}{ds}\phi(s)=-\nabla|Z_{\gamma}(\phi(s))|^2, \\
          & \mbox{Arg}Z_{\gamma}(\phi(s_0))=\vartheta.
        \end{align*}

\subsection{Open Gromov-Witten Invariants on K3 Surfaces}
   We will follow the idea of family Floer homology \cite{F5}\cite{F6} (see also \cite{T4}) and review the construction of the open Gromov-Witten invariants on elliptic K3 surface defined in \cite{L8}. Let $X\rightarrow B$ be an elliptic K3 surface and let $L_u$ denote the fibre over $u\in B$. 
   
   Given $u\in B_0$ and $\gamma\in H_2(X,L_u)$, we denote $W'_{\gamma}$ to be the subset of $B_0$ collecting those $u\in B_0$ such that there exist $\gamma_1,\gamma_2\in H_2(X,L_u)$ such that 
   \begin{enumerate}
       \item   $\langle \gamma_1,\gamma_2\rangle \neq 0$,
       \item  $\gamma=\gamma_1+\gamma_2$, and
       \item   $\mathcal{M}_{\gamma_i}(X_{\vartheta},L_u)\neq 0$ for the same $\vartheta\in S^1$.
    
   \end{enumerate}
   Now let $\gamma\in H_2(X,L_u)$ be a 
   primitive class with $\vartheta=\mbox{Arg}Z_{\gamma}(u)$ and $u\notin W'_{\gamma}$, there exists an affine line $l_{\gamma}$ (with respect to the complex affine structure of $X_{\vartheta}$) passing through $u$ such that $\mbox{Arg}Z_{\gamma}$ is constant along $l_{\gamma}$. Choose a sequence of nested simply connected neighborhood $ U_{i+1}\subseteq \bar{U}_{i+1}\subseteq U_i$ of $u$ and a sequence of pair of points $u_i^{\pm}\in U_i$
     \begin{enumerate}
        \item $\cap_i U_i=\{u\}$.
        \item $Z_{\gamma}(u^+_i)>Z_{\gamma}(u_i)>Z_{\gamma}(u^-_i)$.
        \item $L_{u_i^{\pm}}$ do not bound any holomorphic discs in $X$.
     \end{enumerate} Choose a path $\phi_i$ in $U_i$ from $u^-_i$ to $u_i^+$. Applying the Fukaya's trick, we get an isomorphism 
          \begin{align*}
        F_{u,\phi_i}:H^1(L_u,\Lambda_+)\rightarrow H^1(L_u,\Lambda_+),
          \end{align*}      
    which is independent of the choice of the paths $\phi_i$ by Lemma \ref{521}.        
    Given each $U_i$, there are only finitely $l_{\gamma'}$ that have non-trivial intersection with $U_i$. Together with the assumption $\cap_i U_i=\{u\}$, the limit $F_{u}:=\lim_{i\rightarrow \infty}F_{u,\phi_i}$ exists.       
          
   The explicit form of (\ref{520}) can be computed 
   \begin{thm}
   (Theorem 6.15\cite{L8}) The transformation $\tilde{F}_{u}$ is of the form
     \begin{align} \label{998}
         \tilde{F}_u:z^{\partial\gamma'}\mapsto z^{\partial \gamma'}f_{\gamma}(u)^{\langle \gamma',\gamma\rangle}, 
     \end{align} for some power series $f_{\gamma}(u)\in 1+\Lambda_+[[z^{\partial\gamma}]]$. Here $\langle \gamma',\gamma\rangle$ denotes the intersection pairing of the corresponding boundary classes.
   \end{thm} The theorem motivates the following definition. 
     \begin{definition} \label{1003} Let $X\rightarrow B$ be an elliptic K3 surface with a hyperK\"ahler pair $(\omega,\Omega)$\footnote{Notice that $X$ as a K3 surface already determined $\Omega$ up to a $\mathbb{C}^*$-scaling}. 
        Let $u\in B_0$, $\gamma\in H_2(X,L_u)$ primitive such that $\vartheta=\mbox{Arg}Z_{\gamma}+\frac{\pi}{2}$ and $u\notin W'_{\gamma}$, then the open Gromov-Witten invariant $\tilde{\Omega}(\gamma;u)$ is defined via 
       \begin{align*}
           \log{f_{\gamma}(u)}=\sum_{d\geq 1}d\tilde{\Omega}(d\gamma;u)(T^{\omega(\gamma)}z^{\partial\gamma})^d
       \end{align*} on $X_{\vartheta}$. 
     \end{definition}
 The following are some properties of the open Gromov-Witten invariants:
 \begin{thm} (Theorem 6.24 \cite{L8}) \label{1004}
    Under the same assumption in Definition \ref{1003}, 
     \begin{enumerate}
        \item $\tilde{\Omega}(\gamma;u)$ is independent of the choice of the Ricci-flat metric $\omega$. 
        \item (reality condition) $\tilde{\Omega}(-\gamma;u)=\tilde{\Omega}(\gamma;u)$
        \item If there exists a path connecting $u,u'\in B_0\backslash W'_{\gamma}$ does not intersect $W'_{\gamma}$, then $\tilde{\Omega}(\gamma;u)=\tilde{\Omega}(\gamma;u')$. In particular, $\tilde{\Omega}(\gamma;u)$ is locally constant in $u$ for $u$ away from $W'_{\gamma}$. 
        \item  If there exists a path connecting $u,u'\in B_0\backslash W'_{\gamma}$ and intersecting $W'_{\gamma}$ at a generic point $u_0$, then $\tilde{\Omega}(\gamma;u)=\tilde{\Omega}(\gamma;u')$ unless there exist $\gamma_i\in H_2(X,L_{u_0})$ such that $\gamma=\sum_i\gamma_i$ and $\tilde{\Omega}(\gamma_i;u_0)\neq 0$. 
     \end{enumerate}
 \end{thm}
The last part of the Theorem \ref{1004} gives a hint of the connection of the open Gromov-Witten invariants $\tilde{\Omega}(\gamma;u)$ and tropical geometry, which we will discuss in the Section \ref{1005}. 

\begin{rmk}
	There is another definition of the open Gromov-Witten invariants under the same setting by the author via the "first principle" \cite{L4}. It is later proved in Theorem 6.29 \cite{L8} that the two definitions coincide. In particular, the open Gromov-Witten invariants defined here also have the usual enumerative meaning, the virtual counting of the number of holomorphic discs with special Lagrangian boundary conditions. Here we choose the definition for the purpose of later proofs in the paper. 
\end{rmk}

\subsection{Local Model of Type $I_1$-Singular Fibres}

The Ooguri-Vafa space $X_{OV}$ is an elliptic fibration over a disc $D$ with a type $I_1$ singular fibre over $0\in D$. There exists an $S^1$-action preserving the complex structure. Ooguri-Vafa \cite{OV} constructed an $S^1$-invariant Ricci-flat metric $\omega_{OV}$ on $X_{OV}$ from periodic Gibbons-Hawkings ansatz and known as the Ooguri-Vafa metric. The complex affine structure induced from the special Lagrangian fibration in $(X_{OV})_{\vartheta}$ admits a singularity at the origin. The monodromy of the affine structure is conjugate to $\bigg(\begin{matrix} 1 &1 \\ 0& 1 \end{matrix}\bigg)$ from Picard-Lefschetz formula.


 Let $u\in D$ and $\gamma_e$ be the generator of $H_2(X_{OV},L_u)\cong \mathbb{Z}$. There are  two rays $l_{\pm \gamma_e}$ emanating from the singularity such that the tangents are in the monodromy invariant direction. The two rays $l_{\pm\gamma_e}$ are distinguished in symplectic geometry for the following reason:  the fiber $L_u$ bounds a holomorphic disc in $(X_{OV})_{\vartheta}$ if and only if $u\in l_{\pm\gamma_e}$. Moreover, there is a unique (up to orientation) simple holomorphic disc with boundary on $L_u$ which is the union of vanishing cycles from the singularity along $l_{\gamma_e}$(or $l_{-\gamma_e}$) to $u$ \cite{C}. As $\vartheta$ moves around $S^1$, the ray $l_{\gamma}$ (and $l_{-\gamma}$) rotates counterclockwisely and every point is swept once. In other words, every fibre $L_u$ bounds exactly one simple holomorphic disc of relative class $\gamma$ (and $-\gamma$) in $X_{\vartheta}$ (in $X_{-\vartheta}$) for some $\vartheta\in S^1$. Moreover, the open Gromov-Witten invariants of $X_{OV}$ is calculated \cite{L4}\cite{L8}
  \begin{align*}
      \tilde{\Omega}_{I_1}(\gamma;u)=\begin{cases}
                  \frac{(-1)^{d-1}}{d^2},& \mbox{ if $\gamma=d\gamma_e$}, \\
                  0, &\mbox{otherwise.} 
              \end{cases}
  \end{align*} 
 
 Given an elliptic K3 surface $X$, the tubular neighborhood of a type $I_1$ singular fibre is modeled by the Ooguri-Vafa space. When the K\"ahler class $[\omega]$ is chosen such that $\int_{L_u}[\omega]$ is small enough, then the Ricci-flat metric is $C^0$ close to the Ooguri-Vafa metric. This allows one to use cobordism argument to compute some open Gromov-Witten invariants of $X$ near a type $I_1$ singular fibre\footnote{Notice that the open Gromov-Witten invariant $\tilde{\Omega}(\gamma;u)$ is independent of the choice of the Ricci-flat metric.}.  
   \begin{thm} (Theorem 4.44 \cite{L4})
      Let $u_0\in \Delta$ corresponds to a type $I_1$ singular fibre. Let $\gamma_e$ denote the relative class of the Lefschetz thimble. Then for each $d\in \mathbb{Z}$, there exists an open neighborhood $\mathcal{U}_d\in B$ of $u_0$ such that 
         \begin{align*}
              \tilde{\Omega}(\gamma;u)=\begin{cases}
                              \frac{(-1)^{d-1}}{d^2},& \mbox{ if $\gamma=d'\gamma_e$, $|d'|<d$}, \\
                              0, &\mbox{otherwise,}
                              \end{cases} 
         \end{align*}if $u\in \mathcal{U}_d$ and $|Z_{\gamma}(u)|<d|Z_{\gamma_e}(u)|$.
   \end{thm}
\begin{rmk}
  Sometimes it is more convenient to formally absorb the contribution of the multiple cover by setting 
      \begin{align} \label{667}
                  \Omega(d\gamma;u)=-\sum_{k|d}c(\gamma;u)^{\frac{d}{k}}\mu(k)\frac{\tilde{\Omega}(\frac{d}{k}\gamma;u)}{k^2}.
      \end{align} Here $c:H_2(X,L_u)\rightarrow \{\pm 1\}$ is the quadratic refinement satisfying 
      \begin{align*}
         c(\gamma_1+\gamma_2;u)=(-1)^{\langle \gamma_1,\gamma_2\rangle}c(\gamma_1;u)c(\gamma_2;u)
             \end{align*} and $c(\gamma;u)=-1$ if $\gamma$ is the parallel transport of Lefschetz thimble of a type $I_1$ singular fibre to $u$. The function $\mu$ is the Mobius function. 
\end{rmk}

\section{Tropical Geometry of K3 Surfaces with only Type $I_1$ Singular Fibres} \label{1005}

\begin{definition} \label{914} Fix $\vartheta\in S^1$ generic.
   A tropical disc of $X_{\vartheta}$ ends at $u\in B_0$ is a $3$-tuple $(\phi,T,w)$ satisfying the following properties\footnote{Here we avoid the situation when the tropical disc has an edge contracted to the singularity by considering the affine structure of $B_{\vartheta}$ for a generic $\vartheta$. We refer the readers to Section 4 \cite{L8} for the general case.}:
    \begin{enumerate}
       \item $T$ is a tree with a root $x$ with no valency two vertices.
       \item $\phi:T\rightarrow B$ such that 
          \begin{enumerate}
             \item $\phi(x)=u$ and all the other valency one vertices are mapped to $\Delta$. For all other points on $T$ are mapped into $B_0$.
             \item For each edge $e$, then $\phi(e)$ is either an immersion onto an affine line segment (with respect to the complex affine structure of $X_{\vartheta}$) in $B_0$ or $\phi(e)$ is a point. 
             \item If $e$ is an edge attached to a valency one vertex (corresponding to an $I_1$ singularity) other than $x$, then $\phi(e)$ is in the monodromy invariant direction.
          \end{enumerate}
       \item For each edge $e$, $w(e)\in \Gamma(\phi(e),T_{\mathbb{Z}}B)$ such that the balancing condition holds: let $v$ be a vertex with valency $k+1$, say $e_{out}$ is the edge closest to the root and $e_1,\cdots,e_k$ be the other edges adjacent to $v$. Then 
         \begin{align*}
             w_{e_{out}}(v)=\sum_i w_{e_i}(v).
         \end{align*}
    \end{enumerate}   
\end{definition}
It is a standard question in tropical geometry that given a tropical disc whether there exits a holomorphic disc in the corresponding relative class. The open Gromov-Witten invariants provide such a sufficient condition. 
\begin{thm}
(Theorem 6.25 \cite{L8}) Give $u\in B_0$ generic and let $\gamma\in H_2(X,L_u)$ such that $\tilde{\Omega}(\gamma;u)\neq 0$. Then there exists a tropical disc $(\phi,T,w)$ such that $[\phi]=\gamma$. 
\end{thm}
This gives a sufficient condition for a tropical discs to be lifted to a holomorphic disc. On In other words,
this can be viewed as a weaker version of the correspondence theorem between tropical discs and holomorphic discs.

To further match each counting of holomorphic discs and tropical discs, we first need to associate each tropical disc a relative class. This can be defined by induction of the number of vertices of the tropical discs. If the tropical disc $(\phi,T,w)$ of $X_{\vartheta}$ has only one vertex except the root $x$, then the vertex is mapped to a singularity $u_0$ and the unique edge is mapped to an affine line segment from $u_0$ to $u=\phi(x)$ in the monodromy invariant direction. We set the relative class $[\phi]$ associated to $(\phi,T,w)$ to be the Lefschetz thimble $\gamma_e$ of the singular fibre $L_{u_0}$ (with the orientation such that $\int_{\gamma_e}\omega_{\vartheta}>0$). Assume that we already defined the relative classes of all tropical discs with number of vertices less than $k$ and $(\phi,T,w)$ is a tropical disc with $k$ vertices. Let $e$ be the edge of $T$ adjacent to the root $x$ and $v$ is the other vertex. Then one will get sub tropical discs $(\phi_i,T_i,w_i)$ with stop at $\phi(v)$ be deleting the edge $e$. By induction hypothesis, we already define $[\phi_i]\in H_2(X,L_{\phi(v)})$ for each $(\phi_i,T_i,w_i)$. We will define the relative class $[\phi]\in H_2(X,L_{\phi(x)})$ to be the parallel transport of $\sum_i [\phi_i]$ along $\phi(e)$. 
 To match the counting of tropical discs with the open Gromov-Witten invariants, we need to associate weights to tropical discs with each vertex which has valency less than $4$. 
\begin{definition} Let $(\phi, T, w)$ be a tropical disc with stop at $u\in B_0$ and each vertex of $T$ the valency is at most $3$. Then we define the associate weight to be
   \begin{align} \label{4029}
            \mbox{Mult}(\phi):=\prod_{\substack{v\in C^{int}_0(T)\\ v: trivalent}}\mbox{Mult}_v(\phi)\prod_{v\in C^{ext}_0(T)\backslash \{u\}}\frac{(-1)^{w_v-1}}{w_v^2}  \prod_{T_e: \phi(e) \mbox{is a point}}|\mbox{Aut}(\bold{w}_{T(e)})|^{-1},
        \end{align} where the notation is explained below:
        \begin{enumerate}
           \item Let $e\in C_1(T)$ be a contracted edge and $T_e$ is the connected subtree of $T$ containing $e$.  Let $e_0,e_1,\cdots,e_m \in C_1(T)$ be the edges adjacent to $T_e$ and $e_0$ is the one closest to the root. Denote the weight of $e_i$ by $w_i$. Assume that there are $n$ possible direction of $w(e_i)$ and $w_{ij}$ denotes the number of $e_s$ such that $w(e_s)$ in $i$-th direction and with divisibility $j$.
             Then we set $\bold{w}_{T_e}=(\bold{w}_1,\cdots,\bold{w}_n)$, where $\bold{w}_i=(w_{i1},\cdots,w_{il_i})$. The last product factor in (\ref{4029}) doesn't repeat the factor if $T_e=T_{e'}$. 
           \item For a set of weight vectors $\bold{w}=(\bold{w}_1,\cdots, \bold{w}_n)$ and $\bold{w}_i=(w_{i1},\cdots, w_{il_i})$, for $i=1,\cdots, n$. We set 
             \begin{align*}
                a^i_n=|\{w_{ij}|w_{ij}=n\}|
             \end{align*} and 
                       \begin{align*}
                          b_l=\#\{i|Z_{[\phi_i]}(\phi(e))\in l\}
                       \end{align*} for any ray $l\in \mathbb{C}^*$. Then we define
             \begin{align*}
               |\mbox{Aut}(\bold{w})|=\bigg(\prod_{i}\prod_{n\in \mathbb{N} \atop a^i_n\neq 0} (a^i_n)! \bigg)\prod_{l} (b_l)!,
             \end{align*} where the first factor is the size of the subgroup of the permutation group $\prod_i
             \Sigma_{l_i}$ stabilizing $\bold{w}$.
        \end{enumerate}
\end{definition}

However, to avoid the overcount of tropical discs, one only counts the admissible tropical discs which only allow valency one and three together with some technical conditions with weights. We will refer the readers for the definition of the definition of admissible tropical discs to \cite{L8}.
\begin{definition}
   Given $u\in B_0$ generic and $\gamma\in H_2(X,L_u)$. Then the corresponding weighted count of tropical discs is defined to be
     \begin{align*}
         \tilde{\Omega}^{trop}(\gamma;u)=\sum_{(\phi,T,w):[\phi]=\gamma}\mbox{Mult}(\phi),
     \end{align*} where the summation is over all admissible tropical discs $(\phi,T,w)$ with respect to the affine structure $B_{\mbox{Arg}Z_{\gamma}(u)-\pi/2}$. 
\end{definition}
 With the correct definition, one can match the open Gromov-Witten invariants with the tropical counts.
\begin{thm}
\cite{L8} Let $X$ be an elliptic K3 surface with only type $I_1$ singular fibres. Let $u\in B_0$ be generic and $\gamma \in H_2(X,L_u)$. Then the open Gromov-Witten invariants are captured by the weighted count of tropical discs. Namely,
  \begin{align*}
     \tilde{\Omega}(\gamma;u)=\tilde{\Omega}^{trop}(\gamma;u).
   \end{align*}
\end{thm}

\section{Local Models of Other Singular Fibres}
  The following theorem is well-known among algebraic geometers and we include the proof just for self-containedness.
 \begin{thm} \label{35}
     Let $X$ be an elliptic K3 surface. There exists a deformation to $X'$, which is an elliptic K3 surface with only $I_1$-type singular fibres, via a path of elliptic K3 surfaces.
   \end{thm}
 \begin{proof}
   Assume that $X$ is an elliptic K3 surface with a singular fibre $L_u$ not of type $I_1$. Since the intersection matrix of K3 surface is even, there is no $-1$ curve inside the K3 surface $X$. In particular, the possible singular fibres of elliptic K3 surfaces are classified by Kodaira \cite{K1}. If any singular fibre has more than one components, then $X$ falls in a divisor of the moduli space by Torelli theorem of K3 surfaces and therefore non-generic. Thus, it suffices to prove the case when all the singular fibres of $X$ are irreducible without loss of generality. From the Kodaira's classification, then the possible singular fibres are of type $I_1$ (simple nodal curves) or type $II$ (cusp curves). By Proposition 9.1 \cite{BPV}, the relative Jacobian $\pi_J:J(X)\rightarrow \mathbb{P}^1$ of such elliptic surface is an elliptic surface with a section such that $\pi^{-1}(u)\cong \pi_J^{-1}(u)$ for each $u\in B$. This induces a submersion from the moduli space of marked elliptic K3 surfaces to the moduli space of marked elliptic K3 surfaces with a section. Given an elliptic K3 surface $\pi:X\rightarrow B$ with a section, the fibres of the above submersion are parametrized by $H^1(B,X^{\#})$, where $X^{\#}$ is the sheaf of sections. From the short exact sequence 
       \begin{align*}
          0\rightarrow R^1\pi_*\mathbb{Z}\rightarrow T^*\mathbb{P}^1\rightarrow X^{\#}\rightarrow 0,
       \end{align*} we have $H^1(B,X^{\#})\cong H^1(\mathbb{P}^1,T^*\mathbb{P}^1)/H^1(\mathbb{P}^1,R^1\pi_*\mathbb{Z})$ which is path connected of dimension $1$. Therefore, it suffices to prove the theorem for the relative Jacobians of elliptic K3 surfaces. In this situation, the elliptic surface with a section can be expressed in Weierstrass form
      \begin{align}\label{5}
        y^2z=x^3+a(t)xz^2+b(t)z^3 \in \mathbb{P}( \mathcal{O}_{\mathbb{P}^1}(2)\oplus\mathcal{O}_{\mathbb{P}^1}(3)\oplus\mathcal{O}_{\mathbb{P}^1}),
      \end{align} where $a(t)\in H^0(\mathbb{P}^1,\mathcal{O}_{\mathbb{P}^1}(4))$ and $b(t)\in H^0(\mathbb{P}^1,\mathcal{O}_{\mathbb{P}^1}(6))$. Straight-forward computation shows that the discriminant locus of the elliptic K3 surface in (\ref{5}) is $24$ points for generic choices of $a(t)$ and $b(t)$. This finishes the proof of the theorem.
 \end{proof}
   
 \begin{rmk}
   The statement of Theorem \ref{35} is not always true for general compact elliptic fibrations. There are examples of "non-Higgsable cluster" singular fibres. For instance such example can be constructed by a Weierstrass model $\pi: W_g\rightarrow \mathbb{F}_3$ of an elliptic Calabi-Yau $3$-fold over the Hirzebruch surface $\mathbb{F}_3$ and it admits a non-Higgsable singular fibre of type IV \cite{M4}. On the other hand, the local equation
      \begin{align*}
         y^2=x^3+(s^2+2\epsilon)x+s^2+\epsilon
      \end{align*} gives a deformation of a type $IV$ singular fibre to four type $I_1$ singular fibres.
 \end{rmk}

    The following is the the main theorem to approach the tropical geometry of K3 surfaces.
         \begin{thm} \label{399}
            Let $X_t$ be a $1$-parameter family of elliptic K3 surfaces with $X_0=X$.
            Given any $u\in B_0$, $\gamma\in H_2(X,L_u)$ such that $u\notin W'_{\gamma}$, then there exists $t_0$ such that 
            \begin{align*}
            \tilde{\Omega}_t(\gamma;u)=\tilde{\Omega}(\gamma;u)
            \end{align*} 
            for $t$, $|t|<t_0$. Here $\tilde{\Omega}_t(\gamma;u)$ denotes the open Gromov-Witten invariants of $X_t$. 
         \end{thm}
         \begin{proof}
            Assume that $\mbox{Arg}Z_{\gamma}(u)=\vartheta-\pi/2$, there exists an affine line $l_{\gamma}$ with respect to the complex affine coordinate of $X_{\vartheta}$ passing through $u$ such that $\mbox{Arg}Z_{\gamma}(u')=\vartheta$ for $u'\in l_{\gamma}$. Similarly, let $l^t_{\gamma}$ be an affine ray from the singularity such that $\mbox{Arg}Z_{\gamma}(u)=\vartheta$ along $l^t_{\gamma}$. Since $\lim_{t\rightarrow 0}\Omega_t=\Omega$, we have $l^t_{\gamma}$ is a small deformation of $l_{\gamma}$. Choose $U'\subseteq U$ be a simply connected neighborhood of $u$ but does not contain $0$. Let $t_i\rightarrow 0$ be a sequence of positive numbers, $u_i^{\pm}\in U'$ are sequences of pairs of points on different sides of $l_{\gamma}$ such that 
                            \begin{enumerate}
                              \item $\mbox{Arg}Z_{\gamma}(u_i^+)<\vartheta<\mbox{Arg}Z_{\gamma}(u_i^-)$, and 
                              \item $\lim_{i\rightarrow \infty}u_i^{\pm}=u$. 
                              \item $u^{\pm}_i \notin l_{\gamma}^t$ for any $t\leq t_i$. 
                            \end{enumerate} Choose a sequence of positive numbers $\lambda_i$ such that $\lim_{i\rightarrow \infty}\lambda_i=\infty$. By Gromov compactness theorem, there exists only finitely many relative classes $\gamma^{(i)}_j$ with boundary on $L_{u'}, u'\in U'$ and  $|Z_{\gamma^{(i)}_j}(u)|<\lambda_i$ can be represented by holomorphic discs. We may further choose $U_i\subseteq U'$ (not necessarily connected) avoiding all such $l^t_{\gamma^{(i)}_j}$, $t<t_i$. When $|t|\ll 1$, we can have non-empty $U_i$s and choose $u^{\pm}_i$ from $U_i$. In particular, $\phi_i$ will avoid all such  $l^t_{\gamma^{(i)}_j}$ except $l^t_{\gamma}$. Let $\phi_i^{\pm}$ be the $1$-parameter family of (almost) complex structures induced from $(X_t,L_{u_i^{\pm}})$ From the construction, we have 
                                         \begin{align}
                                            & F_{\phi_i^{\pm}}\equiv \mbox{Id} \mbox{( mod $T^{\lambda_i}$)} 
                                         \end{align}
                                     from the first part of the theorem. Let $\phi$ (and of $\phi_{i,t}$) corresponding to the paths connecting $u^i_{\pm}$ on the base of $X$ (and $X_t$ respectively).
    If $t$ is small enough, the composition of the path $(\phi_i^+)^{-1}\circ (\phi_{i,t})^{-1}\circ \phi^{-}_i\circ{\phi_i}$ is contractible. From Lemma \ref{521}, we have the following 
            \begin{align} \label{46}
              F_{\phi_i}&= (F_{\phi_i^-})^{-1} \circ F_{\phi_{i,t}} \circ F_{\phi_i^+}= F_{\phi_{i,t}} (\mbox{ mod }T^{\lambda_i}) .
            \end{align} 
       Then for large enough $i$, the left hand side of (\ref{46}) defines $\tilde{\Omega}(\gamma;u)$ while the right hand side defines $\tilde{\Omega}_t(\gamma;u)$. This finishes the proof of the theorem. 

         \end{proof}    
       By Proposition 4.4.1 \cite{S6}, any holomorphic disc in $X$ with boundary in a special Lagrangian torus fibre and small symplectic area will fall in a tubular neighborhood of the fibre. If the tubular neighborhood does not contain a singular fibre, then $L_u$ topologically cannot bound any disc and leads to a contradiction. Therefore, any holomorphic disc with small symplectic area falls in a tubular neighborhood of some singular fibre, say $L_0$. Let $U$ be a neighborhood of $0$ and $X_{U}:=p^{-1}(U)$ be the pre-image of the fibration. Since $X_U$ has a deformation retract to the central fibre, we have 
                   $ H_2(X_U)\subseteq H_2(X)$ and this implies that $H_2(X_U,L_u)\subseteq H_2(X,L_u)$. 
       Therefore, we have 
         \begin{align} \label{666}
              \mathcal{M}_{\gamma}(X,L_u)=\mathcal{M}_{\gamma}(X_U,L_u)
          \end{align} and similarly for the $1$-parameter family of moduli spaces considered in the proof of Theorem \ref{399}. In particular, the later moduli space is compact and one can define the open Gromov-Witten invariants of $X_U$ for the pair $(u,\gamma)$, which we will denote by $\tilde{\Omega}^{loc}(\gamma;u)$. Similar to (\ref{667}), we may also define $\Omega^{loc}(\gamma;u)$.
    Similarly, one can define the open Gromov-Witten invariants $\tilde{\Omega}^{loc}_t(\gamma;u)$ for $(X_U)_t$. Similar to the proof of Theorem \ref{399}, one have that 
          \begin{align} \label{312}
            \lim_{t\rightarrow 0}\tilde{\Omega}^{loc}_t(\gamma;u)=\tilde{\Omega}^{loc}(\gamma;u).
          \end{align}
%

        We will use the notation $\tilde{\Omega}_*(\gamma;u)$ for the local open Gromov-Witten invariants of $X_U$, which is the germ of elliptic fibration over a disc with central fibre of type $*=I_n, II, III, IV, I_0^*$ in the later part of the paper.

        It worth mentioning that the argument of the whole paper applies to hyperK\"ahler surfaces $X$ with a minimal elliptic fibration together the following assumptions:
           \begin{enumerate}
              \item There exists a small deformation $X'$ of $X$ through hyperK\"ahler surface with elliptic fibration such that all the singular fibres are of $I_1$-type.
              \item The hyperK\"ahler metric of $X'$ is close enough to Ooguri-Vafa metric near the $I_1$-type singular fibres. 
           \end{enumerate}     
        
   \subsection{Type $I_n$ Singular Fibres} \label{1002}
      \begin{thm} \label{45} There exists primitive relative classes $\gamma_i, i=1,\cdots, n$ such that
        given any $\lambda>0$, there exists a neighborhood $\mathcal{U}_{\lambda}\ni 0$ such that for any  $u\in \mathcal{U}_{\lambda}$, $\gamma \in H_2(X,L_u)$ such that $|Z_{\gamma}(u)|<\lambda$, then for any $d\in \mathbb{Z}$\footnote{This follows from the reality condition in Theorem \ref{1004}.}
           \begin{align*}
              \tilde{\Omega}^{loc}_{I_n}(\gamma;u)=\begin{cases} 
                     \frac{(-1)^{d-1}}{d^2}, & \gamma=d \gamma_i \\
                     0, & otherwise, \end{cases}.
           \end{align*}      
      \end{thm} 
      
      \begin{proof}
          From the gradient estimate of holomorphic discs, holomorphic discs of symplectic area $\epsilon$ should contains in a neighborhood of a singular fibre, say  There are two affine rays $l_{\pm}$ emanating from $0$ such that only special Lagrangian torus fibres over $l_{\pm}$ can bound holomorphic discs with symplectic area less than $\lambda$ in $X_{U}$.

          By Theorem \ref{35}, there exists a $1$-parameter family of hyperK\"ahler structures on the underlying space of $X_U$, which we will denote by $X_{U,\epsilon}$. Then $X_{U,t\neq 0}\rightarrow U$ is an elliptic fibration with $n$ singular $I_1$-type singular fibres. Let $\gamma_i\in H_2(X_{U,t},L_u),i=1,\cdots, n$ be the relative classes of Lefschetz thimbles associate to the $n$ singular $I_1$-type singular fibres. Under the identification $X_{U,t}\cong X_U$, we will view $\gamma_i\in H_2(X_U,L_u)$. The affine line $l_{\gamma_i}^t$ should fall in a neighborhood of $l_{\pm}$. Therefore, we have $\langle \gamma_i,\gamma_j \rangle=0$ and $Z_{\gamma_i}(u)=Z_{\gamma_j}(u)$ for $u\in U$ (here we use the holomorphic volume form on $X_U$ to define the central charge). By Mayer-Vietoris sequence and induction, we have $H_2(X_U,L_u)\cong \mathbb{Z}^{n}$, which is generated by $\gamma_i$. Follow the notation in Theorem \ref{399},
        (See Figure \ref{fig:55} below). 
             \begin{figure}
                                  \begin{center}
                                  \includegraphics[height=3in,width=6in]{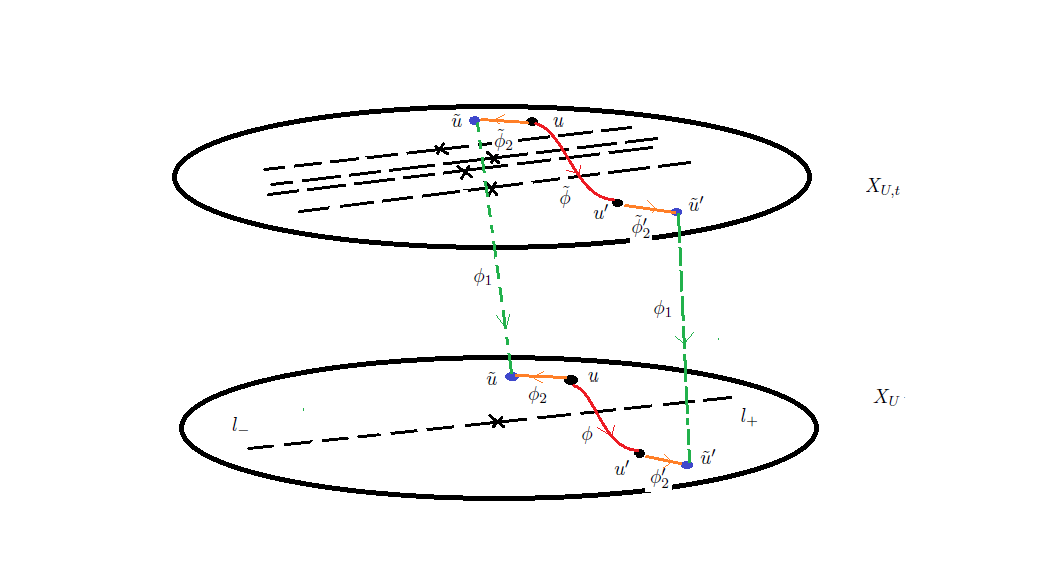}
                                  \caption{Deforming a Type $I_n$ Singular fibre.}
                                   \label{fig:55}
                                  \end{center}
                                  \end{figure}
 $F^{can}_{\tilde{\phi}}$ is the compositions of $n$ copies of the transformation (modulo $T^{\lambda}$) in the form
            \begin{align*}
               & e_1 \mapsto e_1 \\
               & e_2 \mapsto e_2+T^{Z_{\gamma_j}}e_1.
            \end{align*} Thus, we have 
            \begin{align*}
              & F^{can}_{\tilde{\phi}}(e_1)=e_1\\
              & F^{can}_{\tilde{\phi}}(e_2)=e_2+ \sum_{i=1}^{n} T^{Z_{\gamma_i}}e_1. 
            \end{align*} and the theorem is proved.
        
      \end{proof}
     
      \begin{figure}
                                        \begin{center}
                                        \includegraphics[height=3in,width=6in]{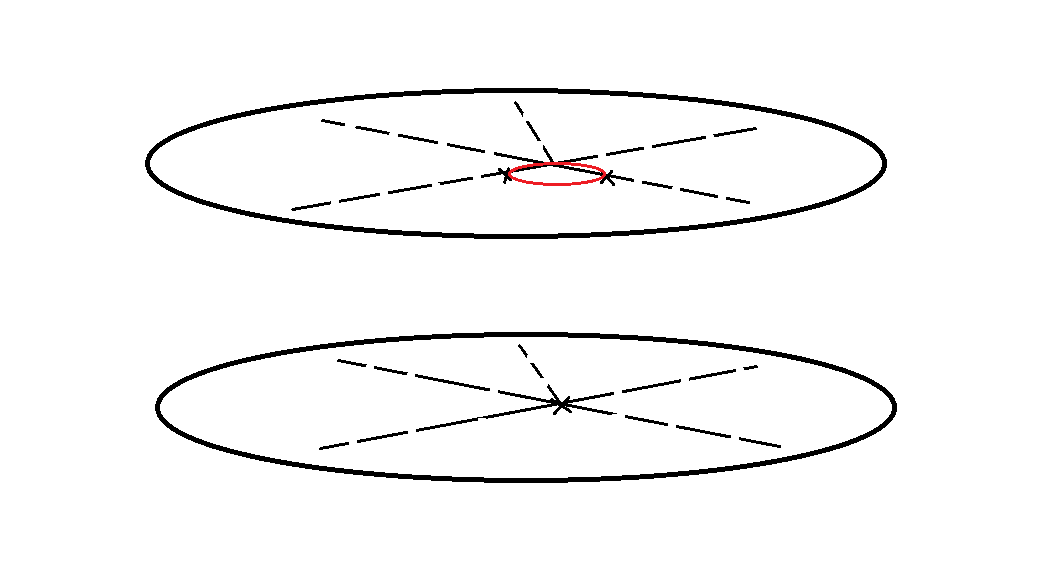}
                                        \caption{Deforming a Type $II$ Singular fibre.}
                                         \label{fig:65}
                                        \end{center}
                                        \end{figure}

 \subsection{Singular Fibres With Monodromy of Finite Order}  
  Recall that the central charge $Z_{\gamma}$ is a multi-value holomorphic function on the puncture disc $D^*$ for any relative class $\gamma$. Assume that the monodromy $M$ around the singular fibre is of finite order (say $k$). Equivalently, this implies the central singular fibre is of type $II$, $III$, $IV$, $IV^*$, $III^*$, $II^*$ or $I_0^*$ in the Kodaira's classification. The corresponding monodromy $M$ are conjugate via $GL(2,\mathbb{Z})$ to 
  $ \bigg(\begin{matrix}
   1 & 1\\ -1 & 0
  \end{matrix}\bigg)$, $ \bigg(\begin{matrix}
     0 & 1\\ -1 & 0
    \end{matrix}\bigg)$, $ \bigg(\begin{matrix}
       0 & 1\\ -1 & -1
      \end{matrix}\bigg)$, $ \bigg(\begin{matrix}
         -1 & -1\\ 1 & 0
        \end{matrix}\bigg)$, $ \bigg(\begin{matrix}
           0 & -1\\ 1 & 0
          \end{matrix}\bigg)$, $ \bigg(\begin{matrix}
             0 & -1\\ 1 & 1
            \end{matrix}\bigg)$ or $ \bigg(\begin{matrix}
               -1 & 0\\ 0 & -1
              \end{matrix}\bigg)$ respectively. Let $\pi_k: z\mapsto z^k$ be the $k$-fold ramification over the disc. Then $\lim_{z\rightarrow 0}Z_{\gamma}(z)=0$ implies that $\pi^*_kZ_{\gamma}$ is a global holomorphic function on the disc $D$. Together with the fact that $\partial \gamma$ and $M(\partial\gamma)$ generate $H_1(L_u;\mathbb{Z})$ and $M$ has no eigenvector over $\mathbb{Z}$ (except the case of $I_0^*$ which we will deal separately in Section \ref{920}), we have 
      \begin{align*}
          \pi_k^*Z_{\gamma}(z)=c_{\gamma}z^{a_*}+o(z^{a_*}), 
      \end{align*} where $z$ is the coordinate on $D$, $a_*\in \mathbb{N}$ such that $(a_*,k)=1$ and $*$ denotes the type of the singular fibre. A priori, $a_*$ might also have dependence on $\gamma$. However, $Z_{\gamma}$ is a homomorphism in $\gamma$ implies that $a_{*}$ is independent of $\gamma$. Equivalently, 
      \begin{align} \label{799}
         Z_{\gamma}(z)=c_{\gamma}z^{\frac{a_*}{k}}+o(z^{\frac{a_*}{k}}).
      \end{align} Notice that $Z_{\gamma}(e^{2 \pi i}u)=Z_{M\gamma}(u)$ and $c_{\gamma}e^{\frac{2\pi ia_*}{k}}=c_{M\gamma}$. The central charge $Z:H_2(X,L_u)\rightarrow \mathbb{C}$
          has kernel the image of $H_2(X)$ thus descends to $H_1(L_u)$. In particular, $\gamma\mapsto c_{\gamma}$ is an surjective homomorphism with the same kernel and descends to $H_1(L_u)$ as well. 
          
    Let $\phi_{\gamma}(t)$ be the affine line ending at $0$ in $B_{\vartheta}$. Then $\phi_{\gamma}(t)$ is characterized by $\mbox{Arg}Z_{\gamma}(\phi_{\gamma}(t))=\vartheta$. From (\ref{799}), we have  
       \begin{align}\label{913}
          \mbox{Arg}Z_{\gamma}(\phi_{\gamma}(t))\sim \mbox{Arg}c_{\gamma}+\frac{a_{*}}{k}\mbox{Arg}\phi_{\gamma}(t),
       \end{align} when $\phi_{\gamma}(t)$ is near $0$. Therefore, $\phi_{\gamma}(t)$ is a straight line in $U$ with respect to the standard affine structure passing through $0$. A direct consequence is the following :
        
        \begin{lem} \label{999}
        If $\gamma$ and $\gamma'$ do not differ by an element in $H_2(X_U,\mathbb{Z})$, then $\phi_{\gamma}(t)$ and $\phi_{\gamma'}(t)$ do not intersect in $U$.
        \end{lem} Let $(X_U)_{\epsilon}$ be an $1$-parameter family of elliptic fibration over $U$ such that $(X_U)_0=X_U$. Denote the corresponding affine line $\phi^{\epsilon}_{\gamma}(t)$. Then similarly we have 
        \begin{lem} \label{311} For $\epsilon\ll 1$, there exists an open neighborhood $U_{\epsilon}\subset U$ such that:
            If $\gamma$ and $\gamma'$ do not differ by an element in $H_2(X_U,\mathbb{Z})$, then $\phi^{\epsilon}_{\gamma}(t)$ and $\phi^{\epsilon}_{\gamma'}(t)$ do not intersect each other in a $U\backslash U_{\epsilon}$, where $\cap U_{\epsilon}=\{0\}$. 
        \end{lem}
          Together with Theorem \ref{1004} and (\ref{913}), we reach the following theorem 
               \begin{thm}
                  Let $M$ be the monodromy around the singular fibre and $M$ of finite order. Then $\tilde{\Omega}^{loc}(\gamma;u)=\tilde{\Omega}^{loc}(M\gamma;u)$, for every $\gamma\in H_2(X,L_u)$. 
               \end{thm}
        
         Theorem \ref{399} allows us to compute the open Gromov-Witten invariants (of a K3 surface) contributed from a germ of a general singular fibre by computing the open Gromov-Witten invariant (of a K3 surface) contributed from specific germ of fibration with only type $I_1$ singularities. A priori, the open Gromov-Witten invariants are not defined all possible deformation of the germ of a general singular fibre due to the issue of existence of hyperK\"ahler metric. However, the tropical discs invariant, which we denote by $\tilde{\Omega}^{trop,loc}(\gamma;u)$, for a germ of elliptic fibration with only $I_1$ type singular fibre can be defined.  Indeed, given any elliptic fibration over a disc with only type $I_1$ singular fibres, there always exists a canonical holomorphic $2$-form\footnote{In particular, if the elliptic fibration is inside an elliptic K3 surface with only $I_1$-singular fibres, this holomorphic $2$-form coincides with the restriction of the holomorphic volume of the K3 surface.} on the total space up to a $\mathbb{C}^*$-scaling. All the tropical discs and the associate weights are determined by the holomorphic $2$-form. 
        The Lemma \ref{311} also leads to the second Theorem to calculate the open Gromov-Witten invariants:
        \begin{thm} Assume that the germ of elliptic fibration $X_U$ has monodromy of finite order, then 
           \begin{align} \label{313}
              \tilde{\Omega}^{loc}(\gamma;u)=\lim_{t\rightarrow 0}\tilde{\Omega}_t^{trop,loc}(\gamma;u).
           \end{align} In particular, one can compute the local open Gromov-Witten invariants by tropical geometry with respect to any deformation (namely, $(X_U)_t$ not necessarily hyperK\"ahler).
        \end{thm}
        \begin{proof}
         From the correspondence theorem between tropical discs and open Gromov-Witten invariant, we have
        \begin{align*}
           \tilde{\Omega}_t^{loc}(\gamma;u)=\tilde{\Omega}^{trop,loc}_t(\gamma;u). 
        \end{align*} Together with (\ref{312}), we have 
         \begin{align*}
            \tilde{\Omega}^{loc}(\gamma;u)=\lim_{t\rightarrow 0}\tilde{\Omega}^{loc}_t(\gamma;u)=\lim_{t\rightarrow 0}\tilde{\Omega}^{trop,loc}_t(\gamma;u),
         \end{align*} for a $1$-parameter family of hyperK\"ahler  elliptic fibration $(X_U)_t$ which admits only type $I_1$ singular fibres for $t>0$ and $(X_U)_0=X_U$. Now we want to prove that right hand side of (\ref{313}) is independent of the family $(X_U)_t$ with only type $I_1$ singular fibre for $t>0$ and $(X_U)_0=X_U$. Notice that Lemma \ref{311} holds for any such family, as $t\ll 1$.  
          Although the curves $l_{\gamma}$ and $l_{\gamma'}$ may intersect when $\gamma-\gamma'\in H_2(X,\mathbb{Z})$, the corresponding transformation in the form of (\ref{998}) commute since $\langle \gamma,\gamma'\rangle=0$. In particular, this implies that 
                $\tilde{\Omega}^{trop,loc}(\gamma;u)$ is independent of $u\in U\backslash U_{\epsilon}$. 
        
        \end{proof}

        One can modified the definition of tropical counting according to Theorem \ref{45}, Theorem \ref{913}, Theorem \ref{916} and Theorem \ref{917}. 
        \begin{thm} Let $X$ be a K3 surface with only type $I_n$, $II$, $III$, $IV$ singular fibres. Assume that $u\notin W'_{\gamma}$, then we have the correspondence theorem 
         \begin{align*}
            \tilde{\Omega}(\gamma;u)=\tilde{\Omega}^{trop}(\gamma;u),
         \end{align*} where $\tilde{\Omega}^{trop}(\gamma;u)$ is defined via Definition \ref{914} with 2.(c) replaced by \\
         (c') If $e$ is an edge adjacent to a valency one vertex corresponds to the type $*$ singular fibre, where $*$ is possibly $I_n$, $II$, $III$, $IV$, then $\phi(e)$ is in an affine segment labeled by $\gamma$ such that $\Omega^{loc}_*(\gamma;u)\neq 0$.
        \end{thm}
         The proof is the same as the proof in Theorem 6.28 \cite{L8} and we will omit it here. 
            
    \subsection{Type $II$ Singular Fibres} \label{22}
  Under the similar notation in Theorem \ref{399} and Section \ref{1002}, assume that $L_0$ is a type $II$ singular fibre over $0\in B$ in $X$ and $U$ is a small open neighborhood $U$ of $0$. Then $X_{U,t}\cong_{top} X_U\rightarrow U$ is an elliptic fibration with two $I_1$-type singular fibres $L_{a_t}$ and $L_{b_t}$. Straight-forward calculation shows that the natural boundary map induces isomorphism $H_2(X_U,L_u)\cong H_1(L_u)$, for any $u\in U$.
     \begin{lem}
       Fix $L_u$ a elliptic fibre and let $\gamma_a,\gamma_b\in H_2(X_{U,t},L_u)$ be the Lefschetz thimble of the two singular fibres, then $\langle \partial\gamma_{a_t},\partial\gamma_{b_t} \rangle=\pm 1$. 
     \end{lem}
     \begin{proof}
       Let $\partial\gamma_1=\partial\gamma_{a_t}, \partial\gamma_2$ be an integral symplectic basis of $H_2(L_u,\mathbb{Z})$ such that $\langle \partial \gamma_1,\partial\gamma_2\rangle=1$. Assume that $\partial \gamma_{b_t}=l_1 \partial \gamma_1+l_2\partial \gamma_2$, then the global monodromy is (up to conjugation)\footnote{Recall that $AB$ is always conjugate to $BA$.} $\bigl(
                 \begin{smallmatrix}
                   1-l_1l_2 & 1-l_1l_2-l_1^2\\
                   -l_2^2 & 1-l_1l_2-l_2^2
                 \end{smallmatrix} \bigr)$ respect to $\partial\gamma_1,\partial \gamma_1$ by the Picard-Lefschetz formula. From the Kodaira's classification of elliptic singular fibres, the trace of the monodromy matrix can only be $0,\pm 1,\pm 2$. Again by comparing the Kodaira's classification, this corresponds to the singular fibres of type $I_2$, $II$ or $II^*$, $I_2^*$ respectively. In particular, this implies $(2l_1+l_2)l_2=1$. Since the trace of the monodromy matrix of a type $II$ singular fibre is $2$, we have $l_2=\pm 1$.  Therefore, a type $II$ singular fibre can only deform into two $I_1$ singular fibres with the pairing	 $\pm 1$. 
     \end{proof}
     
     By Picard-Lefschetz formula, the set $\{\pm\gamma_1,\pm \gamma_2, \pm (\gamma_1+\gamma_2)\}$ is invariant under the monodromy $\bigl(
                                       \begin{smallmatrix}
                                         0 & 1\\
                                         -1 & 1
                                       \end{smallmatrix} \bigr)$.
     Conversely, we fix a basis $\partial \gamma_a,\partial \gamma_b\in H_2(X,L)$ such that the monodromy around the type $II$ singular fibre is $A=\bigl(
                                          \begin{smallmatrix}
                                            0 & 1\\
                                            -1 & 1
                                          \end{smallmatrix} \bigr)$. 
        Assume that $G=\bigl(\begin{smallmatrix}
                                                    a & b\\
                                                    c & d
                                                  \end{smallmatrix} \bigr)\in GL(2,\mathbb{Z})$ commutes with $A$, namely a change of basis preserves the monodromy matrix. 
     Direct computation shows that 
        \begin{align*}
          & G= \bigl(\begin{smallmatrix}
                                                               a & b\\
                                                               -b & a+b
                                                             \end{smallmatrix} \bigr)  \mbox{ and }\\
          & a^2+ab+b^2=1 (detG=\pm 1). 
        \end{align*}          The only possibilities of $G$ are elements of subgroup (isomorphic to $\mathbb{Z}_6$) of $GL(2,\mathbb{Z})$ generated by 
           $\bigl(\begin{smallmatrix}
                      0 & 1\\
                      -1 & 1
              \end{smallmatrix} \bigr)$. It is straightforward to check that the set $\{\pm\gamma_1,\pm \gamma_2, \pm (\gamma_1+\gamma_2)\}$ is preserved under the group action generated by $B$.

     It is known as the "pentagon equation" that $\langle \gamma_1,\gamma_2\rangle=1$, then 
       \begin{align*}
         \mathcal{K}_{\gamma_1}\mathcal{K}_{\gamma_2}=\mathcal{K}_{\gamma_2}\mathcal{K}_{\gamma_1+\gamma_2}\mathcal{K}_{\gamma_1}.
       \end{align*} Following the lines of the proof of Theorem \ref{1001} with Figure \ref{fig:65}, we have the following theorem.      
      \begin{thm} \label{913}
         Under the same notation as in Theorem \ref{45}, 
           \begin{align*}
              \Omega_{II}^{loc}(\gamma;u)=\begin{cases}
                                       1, & \gamma=\pm \gamma_1,\pm\gamma_2,\pm(\gamma_1+\gamma_2),\\
                                       0, & otherwise.
                                      \end{cases}  
           \end{align*} Here $\gamma_1,\gamma_2$ are two relative classes such that the monodromy around the type $II$ singular fibre is  $\bigl(\begin{smallmatrix}
                                                                       0 & 1\\
                                                                       -1 & 1
                                                               \end{smallmatrix} \bigr)$.
      \end{thm}
 
    From the earlier discussion, there are five directions that tropical discs can go into the singularity. This implies that $a_{II}=\frac{5}{6}$ and consecutive ones have the angle differ by $\frac{2\pi}{5}$ with respect to the standard complex structures on $D$ (when we view $D$ as the base of the elliptic fibration). See Figure \ref{fig:90} above. The cut is made to avoid the five affine rays emanating from the singularity locally.
        \begin{figure}
                                               \begin{center}
                                               \includegraphics[height=3in,width=6in]{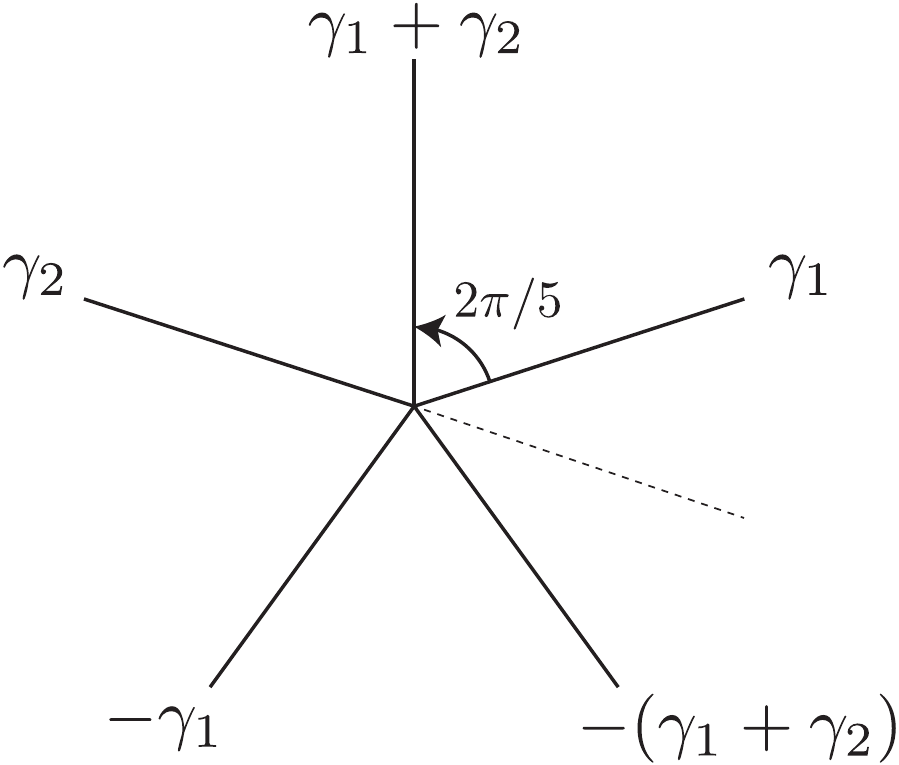}
                                               \caption{Type $II$ Singular fibre. }
                                                \label{fig:90}
                                               \end{center}
                                               \end{figure}

   \subsection{Type $III$ Singular Fibres} \label{919}
     The type $III$ singular fibre is two rational curves intersect at one point of order $2$,  which has Euler characteristic $3$. Therefore, locally an elliptic fibration over a disc with a type $III$ singular fibre as the only singular fibre deform to an elliptic fibration with $3$ type $I_1$ singular fibres. Let $\gamma_1,\gamma_2,\gamma_2$ be the relative classes of Lefschetz thimbles. We have 
         \begin{align*}
          \langle \gamma_1,\gamma_2\rangle=0, \hspace{4mm} \langle \gamma_1,\gamma_3\rangle=\langle\gamma_2,\gamma_3\rangle=1
         \end{align*} and the associate scattering diagram is described in Figure \ref{fig:91}. After choose a suitable $\vartheta$, we may assume that $l_{\gamma_1},l_{\gamma_2},l_{\gamma_3}$ intersect pairwisely. The cut is made to avoid the all possible $l_{\gamma}$s. In particular, we have 
        
     \begin{thm}\label{916} With the notation in Section \ref{919},
        \begin{align*}
                                 \Omega_{III}^{loc}(\gamma;u)=\begin{cases}
                                                         1, & \pm\gamma_1,\pm\gamma_2,\pm\gamma_3, \pm(\gamma_1+\gamma_3),\pm(\gamma_2+\gamma_3),\pm(\gamma_1+\gamma_2+\gamma_3),\\
                                                          0, & otherwise.
                                                         \end{cases}  
                              \end{align*}
     \end{thm}    
     \begin{figure}
                                                     \begin{center}
                                                     \includegraphics[height=3in,width=6in]{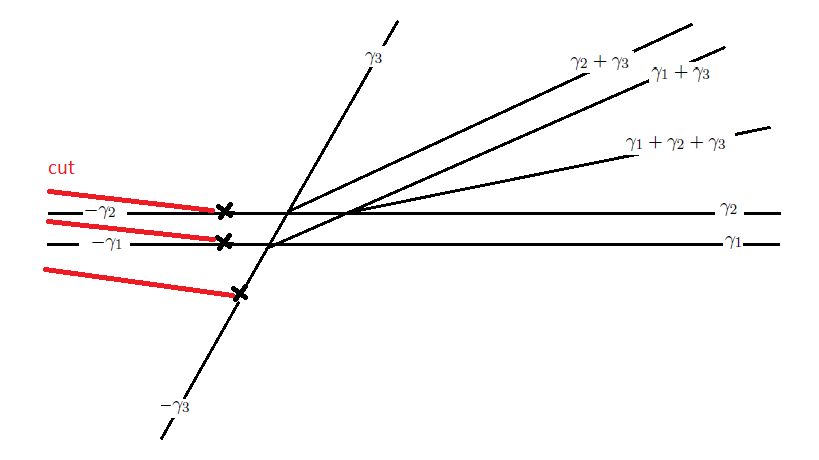}
                                                     \caption{A type $III$ Singular fibre deforms to three type $I_1$ singularities. The total counterclockwise monodromy is $\partial (\gamma_1,\gamma_3)\mapsto \partial(\gamma_1+\gamma_3,-2\gamma_1-\gamma_3)$.}
                                                      \label{fig:91}
                                                     \end{center}
                                                     \end{figure}
             The definition of tropical discs in the presence of type $III$ singular fibres is similar to the Definition \ref{914} with modification similar to that after Theorem \ref{913}.
             
      From the earlier discussion, there are six directions that tropical discs can go into the singularity.  This implies that $a_{III}=\frac{3}{4}$. These six directions are divided into two groups, one are with label in $\{\pm \gamma_1, \pm\gamma_2, \pm(\gamma_1+\gamma_3)\}$ and the other is with label $\{\pm\gamma_3,\pm (\gamma_1+\gamma_2+\gamma_3) \}$. consecutive ones in the same group have angle differ by $\frac{2\pi}{3}$ with respect to the standard complex structures on $D$ (when we view $D$ as the base of the elliptic fibration). Recall that $c_{\gamma}$ is additive with respect to $\gamma$. So if one goes around the origin, the six directions in different groups show up in turn and consecutive directions have angle differ by $\frac{\pi}{3}$. See Figure \ref{fig:100} below. The cut is made to avoid the six affine rays emanating from the singularity. 
              \begin{figure}
                                                     \begin{center}
                                                     \includegraphics[height=3in,width=6in]{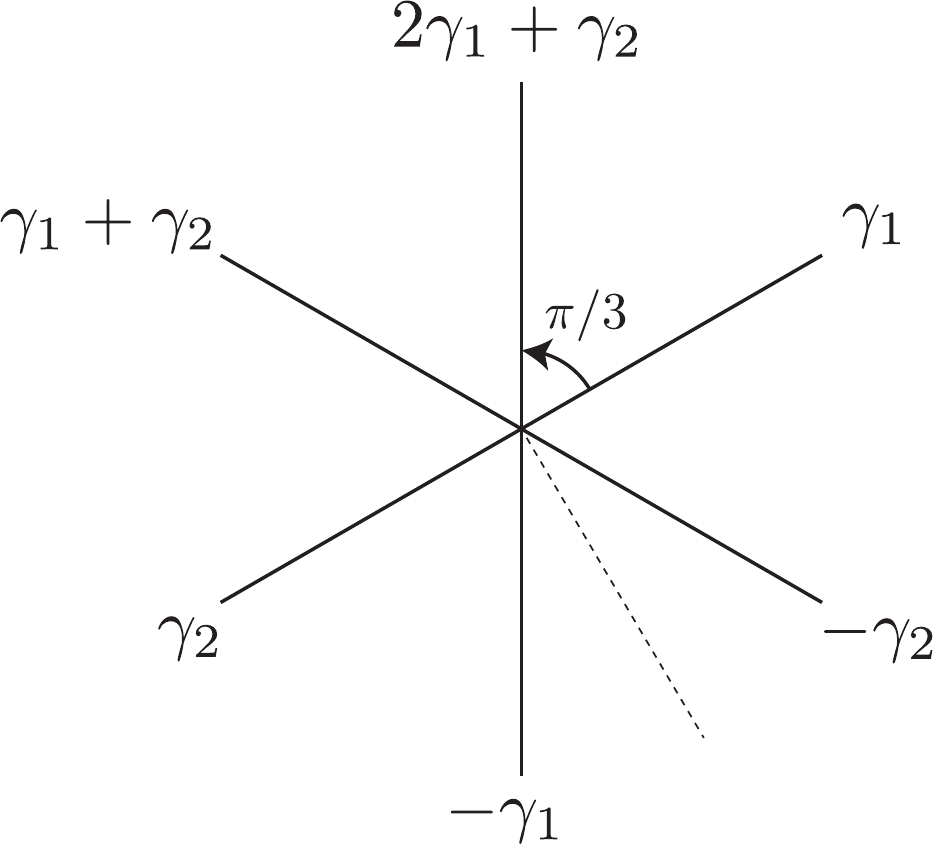}
                                                     \caption{Type $III$ Singular fibre.}
                                                      \label{fig:100}
                                                     \end{center}
                                                     \end{figure}

\begin{rmk}
   The open Gromov-Witten invariants in Theorem \ref{913}, Theorem \ref{916} seems related to the BPS spectrum in weak coupling region in certain examples studied by Gaiotto-Moore-Neitzke \cite{GMN2}. 
\end{rmk}     
   \subsection{Type $IV$ Singular Fibres} \label{918}
     The type $IV$ singular fibre is three rational curves intersect at one point, which has Euler characteristic $4$. Consider the degeneration $\epsilon \rightarrow 0$ of the family
        \begin{align*}
           y^2=x^3+(z^2+\epsilon)x+(z^2+\epsilon).
        \end{align*}  By Tate's algorithm, the family has one singular fibre of type $IV$ when $\epsilon=0$ and the family has two singular fibres of type $II$ when $\epsilon\neq 0$. By Theorem \ref{913}, there are five rays coming out from each type $II$ singularity (see Figure \label{fig:90}). By choosing a particular $\vartheta$, one of the five rays coming out from a type $II$ singularity goes into another. Let $\gamma_1,\gamma_2$ be two relative classes such that the monodromy around the first type $II$ singular fibre is $\bigg(\begin{matrix} 0 & 1\\ -1 & 1\end{matrix}\bigg)$ with respect to the basis and the ray going to another type $II$ singularity labeled by $\gamma_1+\gamma_2$. Assume that the monodromy around the second type $II$ singular fibre is $A$, then $A\bigg(\begin{matrix} 0 & 1\\ -1 & 1\end{matrix}\bigg)$ is conjugate to $\bigg(\begin{matrix} -1 & 1\\ -1 & 0\end{matrix}\bigg)$. Direct computation shows that there is a unique solution that $A=\bigg(\begin{matrix} 0 & 1\\ -1 & 1\end{matrix}\bigg)$ and thus the BPS ray label by $\gamma_1+\gamma_2$ from the first type $II$ singularity overlaps the ray labeled by $-(\gamma_1+\gamma_2)$ from the other type $II$ singularity. We may slightly deform $\vartheta$ and the corresponding scattering diagram is described in Figure \ref{fig:99}. The cut are made to be parallel to $l_{\pm(\gamma_1+\gamma_2)}$. Straightforward calculation shows that $\gamma_1,\gamma_2$ are the unique basis of $H_1(L_u)$ such that the monodromy around the type $IV$ singular fibre is given by $\bigg(\begin{matrix}
             -1 & 1 \\ -1 & 0 
        \end{matrix}\bigg)$ (up to a conjugation of the power of $\bigg(\begin{matrix} 0 & 1\\ -1 & 1\end{matrix}\bigg)$).
         \begin{figure}
                                                \begin{center}
                                                \includegraphics[height=3in,width=6in]{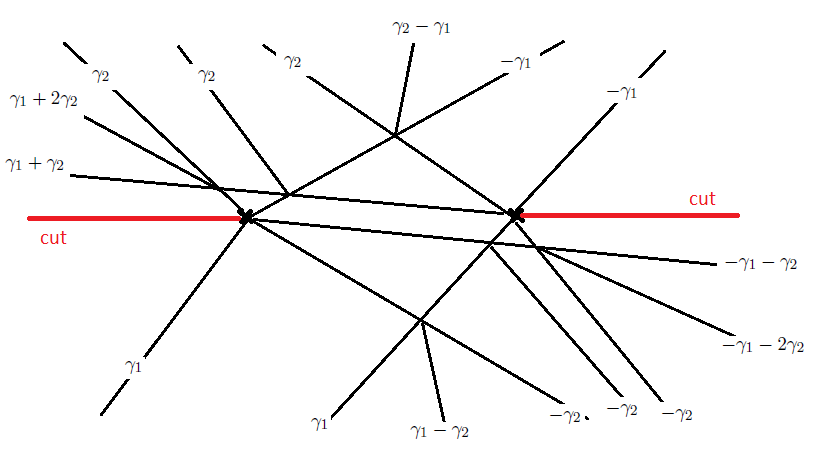}
                                                \caption{A type $IV$ singularity deforms to two type $II$ singularities. The total counterclockwise monodromy is $\partial(\gamma_1,\gamma_2)\mapsto \partial(-\gamma_1-\gamma_2,\gamma_1)$.}
                                                 \label{fig:99}
                                                \end{center}
                                                \end{figure}
        \begin{thm} \label{917} 
        With the notation in Section \ref{918},
            \begin{align*}
                         \Omega_{IV}^{loc}(\gamma;u)=\begin{cases}
                                                 3, & \gamma=\pm \gamma_1,\pm\gamma_2,\pm(\gamma_1+\gamma_2),\\
                                                 1, & \gamma=\pm (\gamma_1-\gamma_2),\pm(\gamma_1+2\gamma_2),\pm (2\gamma_1+\gamma_2),\\
                                                  0, & otherwise.
                                                 \end{cases}  
                      \end{align*}
        \end{thm}
             
          From the earlier discussion, there are eight directions that tropical discs can go into the singularity.  This implies that $a_{IV}=\frac{4}{3}$. These eight directions are divided into two groups, one are with label in $\{\pm \gamma_1, \pm \gamma_2, \pm (\gamma_1+\gamma_2)\}$ and the other is with label $\{\pm(\gamma_1-\gamma_2),\pm (\gamma_1+2\gamma_2), \pm(2\gamma_1+\gamma_2) \}$. Consecutive ones in the same group have angle differ by $\frac{\pi}{2}$ with respect to the standard complex structures on $D$ (when we view $D$ as the base of the elliptic fibration). Recall that $c_{\gamma}$ is additive with respect to $\gamma$. So if one goes around the origin, the eight directions in different groups show up in turn and consecutive directions have angle differ by $\frac{\pi}{4}$.  See Figure \ref{fig:101} above. The cut is made to avoid the $8$ affine rays emanating from the singularity. 
                  \begin{figure}
                                                         \begin{center}
                                                         \includegraphics[height=3in,width=6in]{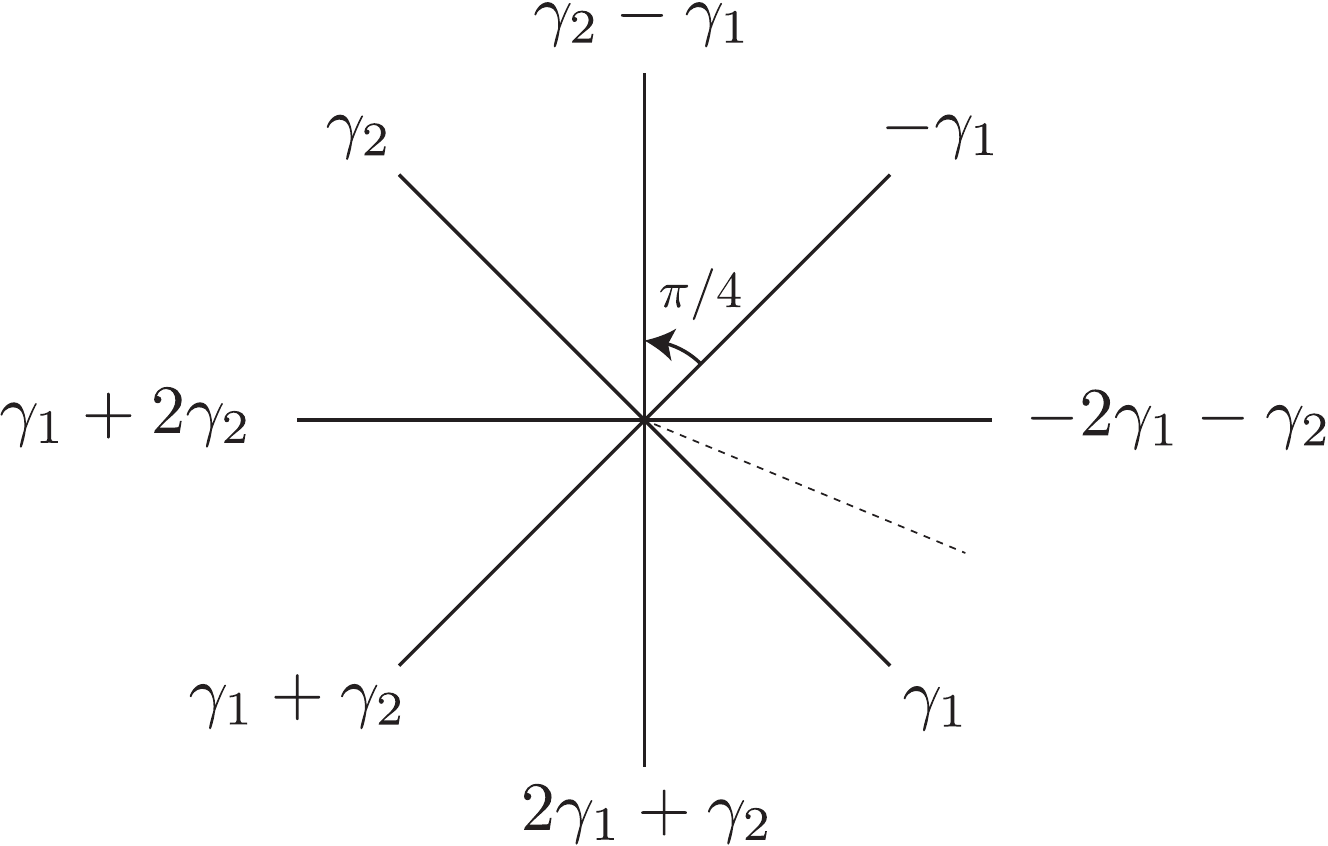}
                                                         \caption{Type $IV$ Singular fibre.}
                                                          \label{fig:101}
                                                         \end{center}
                                                         \end{figure}   
    
   \subsection{Type $I_0^*$ Singular Fibers}\label{920}
 The $I_0^*$ locally arises from the following construction: consider the local non-singular fibration $Y=\mathbb{C}\times D/\Lambda$ where $\Lambda=\mathbb{Z}+\mathbb{Z}\tau(u)$ is a rank $2$ lattice such that $\tau(u)=\tau+u^{2k}$, $k\in \mathbb{Z}$. Then the map $\iota:(x,u)\rightarrow (-x,-u)$ defines an involution on $Y$ and the quotient $X_0=Y/\iota$ is an elliptic fibration with the pillow case as the central fibre. Let $X$ be the blow up the four $A_1$ singularities of the central fibre and $X$ gives the general elliptic fibration with an $I_0^*$ singular fibre. We may deform $\tau$ and again by (\ref{913}) and Theorem \ref{1004}, we have the following theorem:      
  \begin{thm} 
     Let $A\in SL_2(\mathbb{Z})$ and $\gamma$ is a relative class. Then 
        \begin{align*}
           \tilde{\Omega}_{I_0^*}^{loc}(\gamma;u)=\tilde{\Omega}_{I_0^*}^{loc}(A\gamma;u)   .     
       \end{align*} In particular, the local open Gromov-Witten invariant of $\gamma$ only depends on the divisibility of $\gamma$. 
  \end{thm}

Boston University\\
E-mail address: yslin@bu.edu

\end{document}